\documentclass[10pt,reqno]{amsart}
\usepackage{latexsym,amsmath,amssymb,amscd}
\usepackage[all]{xy}

\def\today{\ifcase \month \or
   January \or February \or March \or April \or
   May \or June \or July \or August \or
   September \or October \or November \or December \fi
   \space\number\day , \number\year}
   
\makeatletter
  \newcommand\@dotsep{4.5}
  \def\@tocline#1#2#3#4#5#6#7{\relax
     \ifnum #1>\c@tocdepth 
     \else
     \par \addpenalty\@secpenalty\addvspace{#2}%
     \begingroup \hyphenpenalty\@M
     \@ifempty{#4}{%
     \@tempdima\csname r@tocindent\number#1\endcsname\relax
        }{%
         \@tempdima#4\relax
           }%
      \parindent\z@ \leftskip#3\relax \advance\leftskip\@tempdima\relax
      \rightskip\@pnumwidth plus1em \parfillskip-\@pnumwidth
       #5\leavevmode\hskip-\@tempdima #6\relax
       \leaders\hbox{$\m@th
       \mkern \@dotsep mu\hbox{.}\mkern \@dotsep mu$}\hfill
       \hbox to\@pnumwidth{\@tocpagenum{#7}}\par
       \nobreak
        \endgroup
         \fi}
\makeatother 

\begin{document}


\makeatletter
\@addtoreset{figure}{section}
\def\thefigure{\thesection.\@arabic\c@figure}
\def\fps@figure{h,t}
\@addtoreset{table}{bsection}

\def\thetable{\thesection.\@arabic\c@table}
\def\fps@table{h, t}
\@addtoreset{equation}{section}
\def\theequation{
\arabic{equation}}
\makeatother

\newcommand{\bfi}{\bfseries\itshape}

\newtheorem{theorem}{Theorem}
\newtheorem{acknowledgment}[theorem]{Acknowledgment}
\newtheorem{corollary}[theorem]{Corollary}
\newtheorem{definition}[theorem]{Definition}
\newtheorem{example}[theorem]{Example}
\newtheorem{lemma}[theorem]{Lemma}
\newtheorem{notation}[theorem]{Notation}
\newtheorem{problem}[theorem]{Problem}
\newtheorem{proposition}[theorem]{Proposition}
\newtheorem{remark}[theorem]{Remark}
\newtheorem{setting}[theorem]{Setting}

\newtheorem{C}[theorem]{Comment}

\numberwithin{theorem}{section}
\numberwithin{equation}{section}

\renewcommand{\1}{{\bf 1}}
\newcommand{\Ad}{{\rm Ad}}
\newcommand{\Alg}{{\rm Alg}\,}
\newcommand{\Aut}{{\rm Aut}\,}
\newcommand{\ad}{{\rm ad}}
\newcommand{\Borel}{{\rm Borel}}
\newcommand{\Ci}{{\mathcal C}^\infty}
\newcommand{\Cpol}{{\mathcal C}^\infty_{\rm pol}}
\newcommand{\Der}{{\rm Der}\,}
\newcommand{\Diff}{{\rm Diff}\,}
\newcommand{\de}{{\rm d}}
\newcommand{\ee}{{\rm e}}
\newcommand{\End}{{\rm End}\,}
\newcommand{\ev}{{\rm ev}}
\newcommand{\id}{{\rm id}}
\newcommand{\ie}{{\rm i}}
\newcommand{\GL}{{\rm GL}}
\newcommand{\gl}{{{\mathfrak g}{\mathfrak l}}}
\newcommand{\Hom}{{\rm Hom}\,}
\newcommand{\Img}{{\rm Im}\,}
\newcommand{\Ind}{{\rm Ind}}
\newcommand{\Ker}{{\rm Ker}\,}
\newcommand{\Lie}{\text{\bf L}}
\newcommand{\m}{\text{\bf m}}
\newcommand{\pr}{{\rm pr}}
\newcommand{\Ran}{{\rm Ran}\,}
\renewcommand{\Re}{{\rm Re}\,}
\newcommand{\redrtimes}{\,\widetilde{\rtimes}\,}
\newcommand{\redtimes}{\,\widetilde{\times}\,}
\newcommand{\Si}{{\mathcal S}^\infty}
\newcommand{\Sym}{{\rm Sym}\,}
\newcommand{\shg}{{{\mathfrak s}{\mathfrak h}}}
\newcommand{\so}{\text{so}}
\newcommand{\spa}{{\rm span}\,}
\newcommand{\Tr}{{\rm Tr}\,}
\newcommand{\Op}{{\rm Op}}
\newcommand{\U}{{\rm U}}
\newcommand{\Wig}{{\mathcal W}}

\newcommand{\CC}{{\mathbb C}}
\newcommand{\HH}{{\mathbb H}}
\newcommand{\RR}{{\mathbb R}}
\newcommand{\TT}{{\mathbb T}}

\newcommand{\Ac}{{\mathcal A}}
\newcommand{\Bc}{{\mathcal B}}
\newcommand{\Cc}{{\mathcal C}}
\newcommand{\Dc}{{\mathcal D}}
\newcommand{\Ec}{{\mathcal E}}
\newcommand{\Fc}{{\mathcal F}}
\newcommand{\Hc}{{\mathcal H}}
\newcommand{\Jc}{{\mathcal J}}
\newcommand{\Lc}{{\mathcal L}}
\renewcommand{\Mc}{{\mathcal M}}
\newcommand{\Nc}{{\mathcal N}}
\newcommand{\Oc}{{\mathcal O}}
\newcommand{\Pc}{{\mathcal P}}
\newcommand{\Sc}{{\mathcal S}}
\newcommand{\Tc}{{\mathcal T}}
\newcommand{\Vc}{{\mathcal V}}
\newcommand{\Uc}{{\mathcal U}}
\newcommand{\Yc}{{\mathcal Y}}
\newcommand{\Zc}{{\mathcal Z}}

\newcommand{\Bg}{{\mathfrak B}}
\newcommand{\Fg}{{\mathfrak F}}
\newcommand{\Gg}{{\mathfrak G}}
\newcommand{\Ig}{{\mathfrak I}}
\newcommand{\Jg}{{\mathfrak J}}
\newcommand{\Lg}{{\mathfrak L}}
\newcommand{\Pg}{{\mathfrak P}}
\newcommand{\Sg}{{\mathfrak S}}
\newcommand{\Xg}{{\mathfrak X}}
\newcommand{\Yg}{{\mathfrak Y}}
\newcommand{\Zg}{{\mathfrak Z}}

\newcommand{\ag}{{\mathfrak a}}
\newcommand{\bg}{{\mathfrak b}}
\newcommand{\cg}{{\mathfrak c}}
\newcommand{\dg}{{\mathfrak d}}
\renewcommand{\gg}{{\mathfrak g}}
\newcommand{\hg}{{\mathfrak h}}
\newcommand{\kg}{{\mathfrak k}}
\newcommand{\mg}{{\mathfrak m}}
\newcommand{\n}{{\mathfrak n}}
\newcommand{\og}{{\mathfrak o}}
\newcommand{\pg}{{\mathfrak p}}
\newcommand{\sg}{{\mathfrak s}}
\newcommand{\tg}{{\mathfrak t}}
\newcommand{\ug}{{\mathfrak u}}
\newcommand{\zg}{{\mathfrak z}}

\newcommand{\ZZ}{\mathbb Z}
\newcommand{\NN}{\mathbb N}
\newcommand{\BB}{\mathbb B}

\newcommand{\ep}{\varepsilon}

\newcommand{\hake}[1]{\langle #1 \rangle }

\newcommand{\scalar}[2]{\langle #1 ,#2 \rangle }
\newcommand{\vect}[2]{(#1_1 ,\ldots ,#1_{#2})}
\newcommand{\norm}[1]{\Vert #1 \Vert }
\newcommand{\normrum}[2]{{\norm {#1}}_{#2}}

\newcommand{\upp}[1]{^{(#1)}}
\newcommand{\p}{\partial}

\newcommand{\opn}{\operatorname}
\newcommand{\slim}{\operatornamewithlimits{s-lim\,}}
\newcommand{\sgn}{\operatorname{sgn}}

\newcommand{\seq}[2]{#1_1 ,\dots ,#1_{#2} }
\newcommand{\loc}{_{\opn{loc}}}

\makeatletter
\title[Weyl-Pedersen calculus for some semidirect products]{Weyl-Pedersen calculus for some semidirect products of nilpotent Lie groups}
\author{Ingrid Belti\c t\u a, Daniel Belti\c t\u a, and Mihai Pascu}
\address{Institute of Mathematics ``Simion Stoilow'' 
of the Romanian Academy, 
P.O. Box 1-764, Bucharest, Romania}
\email{ingrid.beltita@gmail.com, Ingrid.Beltita@imar.ro}
\address{Institute of Mathematics ``Simion Stoilow'' 
of the Romanian Academy, 
P.O. Box 1-764, Bucharest, Romania}
\email{beltita@gmail.com,  Daniel.Beltita@imar.ro}
\address{%
University ``Petrol-Gaze'' of Ploie\c sti 
and
Institute of Mathematics
``Simion Stoilow''
of the Romanian Academy, 
P.O. Box 1-764, 
Bucharest
Romania}
\email{Mihai.Pascu@imar.ro}
\thanks{This research has been partially supported by the Grant
of the Romanian National Authority for Scientific Research, CNCS-UEFISCDI,
project number PN-II-ID-PCE-2011-3-0131.}
\date{\today}
\makeatother

\begin{abstract} 
For certain nilpotent real Lie groups constructed as semidirect products, 
algebras of invariant differential operators on some coadjoint orbits  
are used in the study of boundedness properties of 
the Weyl-Pedersen calculus of their corresponding unitary irreducible representations. 
Our main result is applicable to all unitary irreducible representations of arbitrary 3-step nilpotent Lie groups.  
\end{abstract}

\maketitle


\section{Introduction}  

Weyl quantization for representations of Lie groups of various types has been an area of quite active research 
(see for instance \cite{Ca07}, \cite{Ca13} and the references therein). 
In this connection, the aim of the present paper is to study $L^2$-boundedness properties of 
a certain operator calculus for unitary representations of nilpotent Lie groups 
of lower nilpotence step, including all 3-step nilpotent Lie groups.  
More specifically, the deep work of 
N.V.~Pedersen (\cite{Pe94}) shows that for each unitary irreducible representation 
of a nilpotent Lie group one can set up a Weyl correspondence between tempered distributions 
on any coadjoint orbit and unbounded linear operators in the representation space 
of any unitary irreducible representation associated with that coadjoint orbit. 
This correspondence has remarkable properties and in the special case of the Heisenberg group 
it recovers the pseudo-differential Weyl calculus which has been widely used in the theory of PDEs. 

The construction of the Weyl-Pedersen calculus is to some extent noncanonical, 
in that it depends on the choice of a Jordan-H\"older basis 
in the nilpotent Lie algebra under consideration. 
Nevertheless, the construction is fully canonical 
in the case of the irreducible representations that are square integrable modulo the center, 
that is, for the representations associated to the flat coadjoint orbits, 
and one can then even prove that it is covariant with respect to the coadjoint action 
(see \cite{BB11}). 
For the representations of this type, we also recently  
characterized the maximal space of smooth functions on the coadjoint orbit 
that is invariant under the coadjoint action and gives rise to bounded linear operators via the Weyl-Pedersen calculus 
(see \cite{BB12}). 
The characterization involves growth properties of the functions 
expressed in terms of the differential operators on the coadjoint orbit 
which are invariant under the coadjoint action. 
It  is then natural to ask about the coadjoint orbits which are not necessarily flat. 
In the present paper we will take a first step in that direction, as we will now explain. 

For the purposes of the irreducible representation theory, one may restrict the attention to the 
groups with 1-dimensional center   
and moreover it is reasonable to begin by considering the groups of a lower nilpotence step. 
Note that the 2-step nilpotent Lie algebras having 1-dimensional center are precisely the Heisenberg algebras, 
and the coadjoint orbits of the Heisenberg groups are always flat. 
So the first nontrivial case to consider is the one mentioned in the following problem: 

\begin{problem}\label{probl2}
Investigate the $L^2$-boundedness properties of the Weyl-Pedersen calculus on coadjoint orbits of 
3-step nilpotent Lie groups. 
\end{problem}

We will address that problem in Corollary~\ref{3step} below, by relying on some results announced in \cite{BBP13}.  
Specifically, we will prove that the aforementioned result from \cite{BB12} carries over 
to 
coadjoint orbits of any 3-step nilpotent Lie group, 
thereby establishing a version of the Calder\'on-Vaillancourt theorem for the Weyl-Pedersen calculus, 
in which the role of the partial derivatives is played by invariant differential operators on the coadjoint 
orbits of 3-step nilpotent Lie groups. 

\subsection*{General notation}
Throughout this paper, Lie groups are finite-dimensional and are denoted by upper case Roman letters,  
and their corresponding Lie algebras are denoted by the corresponding lower case Gothic letters. 
For any real Lie algebra $\gg$ we denote by $U(\gg)$ 
the complexification of the universal associative enveloping algebra of~$\gg$ 
(see for instance \cite{CG90}), 
by $\Aut\gg$ the automorphism group of $\gg$, and by $\Der\gg$ the Lie algebra of all derivations of~$\gg$. 

For any set $A$, the notation $B\subset A$ means that $B$ is a subset of $A$ and $B\ne A$.

\section{Weyl-Pedersen calculus}

In this preliminary section we recall the Weyl-Pedersen calculus set forth in \cite{Pe94}.  
More details on this calculus, additional references, and connections with other operator calculi 
can be found for instance in \cite{BBP13}. 

Let $\gg$ be any nilpotent Lie algebra of dimension $m\ge 1$ 
with its corresponding nilpotent Lie group $G=(\gg,\cdot)$ whose multiplication 
is defined by the Baker-Campbell-Hausdorff fomula. 
Select any Jordan-H\"older basis 
$\{X_1,\dots,X_m\}$ in $\gg$.   
So for $j=1,\dots,m$ if we define $\gg_j:=\spa\{X_1,\dots,X_j\}$ then $[\gg,\gg_j]\subseteq\gg_{j-1}$, 
where $\gg_0:=\{0\}$. 
Let  $\pi\colon G\to\Bc(\Hc)$ be a unitary representation 
associated with a coadjoint orbit $\Oc\subseteq\gg^*$ via Kirillov's correspondence. 
Pick $\xi_0\in\Oc$ with its corresponding coadjoint isotropy algebra $\gg_{\xi_0}:=\{X\in\gg\mid[X,\gg]\subseteq\Ker\xi_0\}$, 
and define the set of jump indices 
$$e:=\{j\mid X_j\not\in\gg_{j-1}+\gg_{\xi_0}\}.$$ 
Denoting $\gg_e:=\spa\{X_j\mid j\in e\}$,  
one has $\gg=\gg_{\xi_0}\dotplus\gg_e$, and the mapping $\Oc\to\gg_e^*$, $\xi\mapsto\xi\vert_{\gg_e}$, 
is a diffeomorphism. 
Hence one can define an orbital Fourier transform $\Sc'(\Oc)\to\Sc'(\gg_e)$,   $a\mapsto\widehat{a}$ 
which is a linear topological isomorphism  and for every $a\in\Sc(\Oc)$ one has  
$$(\forall X\in\gg_e)\quad \widehat{a}(X)=\int\limits_{\Oc}\ee^{-\ie\langle\xi,X\rangle}a(\xi)\de\xi.$$
Here we have the Lebesgue measure $\de x$ on $\gg_e$ corresponding to the basis $\{X_j\mid j\in e\}$ 
and $\de\xi$ is the Borel measure on $\Oc$ for which 
the above diffeomorphism $\Oc\to\gg_e^*$ is a measure preserving mapping 
and the Fourier transform $L^2(\Oc)\to L^2(\gg_e)$ 
is unitary. 
The inverse of this orbital Fourier transform is denoted by $a\mapsto\check{a}$. 
 
With the notation above, 
the \emph{Weyl-Pedersen calculus}  associated to the unitary irreducible representation $\pi$ is 
$$\Op_{\pi}\colon\Sc(\Oc)\to\Bc(\Hc),\quad 
\Op_{\pi}(a)=\int\limits_{\gg_e}\widehat{a}(x)\pi(x)\de x. $$
The space of smooth vectors $\Hc_\infty:=\{v\in\Hc\mid\pi(\cdot)v\in\Ci(\gg,\Hc)\}$ 
is dense in $\Hc$ and has the natural topology of a nuclear Fr\'echet space 
with the space of the antilinear functionals denoted by $\Hc_{-\infty}:={\Hc}_\infty^*$ 
(with the strong dual topology). 
One can show that the Weyl-Pedersen calculus extends to a linear bijective mapping 
$$\Op_{\pi} \colon\Sc'(\Oc)\to\Lc(\Hc_\infty,\Hc_{-\infty}),\quad 
(\Op_{\pi} (a) v\mid w)=\langle \widehat{a},(\pi(\cdot)v\mid w)\rangle$$
for $a\in\Sc'(\Oc)$, $v,w\in\Hc_\infty$, where in the left-hand side  $(\cdot\mid\cdot)$ denotes 
the extension of the scalar product of $\Hc$ to the sesquilinear duality pairing 
between $\Hc_\infty$ and $\Hc_{-\infty}$. 

It turned out in \cite{BB12} and \cite{BBP13} that 
the invariant differential operators on coadjoint orbits are 
a quite efective tool in the study of boundedness properties of the Weyl-Pedersen calculus. 
Specifically, one defines $\Diff(\Oc)$ as the space of all linear differential operators $D$ on $\Oc$ 
that are \emph{invariant} to the coadjoint action, in the sense that   
 $$
(\forall x\in\gg)(\forall a\in C^\infty(\Oc))\quad  
D(a\circ\Ad_G^*(x)\vert_{\Oc})=(Da)\circ\Ad_G^*(x)\vert_{\Oc}.$$
By using these differential operators one can introduce the Fr\'echet space of symbols
\begin{equation*}
\Cc^\infty_b(\Oc) =\{a\in C^\infty(\Oc)\mid  Da\in L^\infty(\Oc)\;\text{for all} \; D\in\Diff(\Oc)\},
\end{equation*} 
with the topology given by the seminorms $\{a\mapsto\Vert Da\Vert_{L^\infty(\Oc)}\}_{D\in\Diff(\Oc)}$. 
This space of symbols will be needed in our main result (Theorem~\ref{main3} below).

\section{Some properties of coadjoint orbits}

In this section we study  
some simple properties of coadjoint orbits (Propositions \ref{proj}~and~\ref{new}) that 
we were not able to find in the literature in this degree of generality, 
so we include full details of their proofs. 
These facts will be needed in the proof of Theorem~\ref{main3}. 

\begin{proposition}\label{proj}
Let $G$ be a Lie group with the Lie algebra $\gg$ and assume that 
we have a closed normal subgroup $G_0\unlhd G$ whose Lie algebra is denoted by $\gg_0$. 
Pick $\xi\in\gg^*$ with its coadjoint orbit $\Oc_{\xi}^G:=\Ad_G^*(G)\xi$ 
and the $G$-coadjoint isotropy group 
$$G_\xi:=\{g\in G\mid\Ad_G^*(g)\xi=\xi\}.$$ 
Set $\xi_0:=\xi\vert_{\gg_0}\in\gg_0^*$ and consider its coadjoint orbit 
$\Oc_{\xi_0}^{G_0}:=\Ad_{G_0}(G_0)\xi_0$ and the $G_0$-coadjoint isotropy group 
$$(G_0)_{\xi_0}:=\{g\in G_0\mid\Ad_{G_0}^*(g)\xi_0=\xi_0\}.$$ 
Define $\rho\colon\gg^*\to\gg_0^*$, $\eta\mapsto\eta\vert_{\gg_0}$.  
Then the following assertions are equivalent: 
\begin{enumerate}
\item\label{proj_item1} 
The restriction mapping $\rho$ gives a well-defined bijection $\Oc_\xi^G\to\Oc_{\xi_0}^{G_0}$. 
\item\label{proj_item2} 
We have: 
\begin{enumerate}
\item\label{proj_item2a} 
$(G_0)_{\xi_0}=G_0\cap G_\xi$; 
\item\label{proj_item2b} 
$G=G_0\cdot G_\xi$. 
\end{enumerate}
\item\label{proj_item3} 
The restriction mapping $\rho$ gives a well-defined diffeomorphism $\Oc_\xi^G\to\Oc_{\xi_0}^{G_0}$. 
\end{enumerate} 
\end{proposition}

\begin{proof}
It is clear that \eqref{proj_item3}$\Rightarrow$\eqref{proj_item1}, 
so it suffices to prove \eqref{proj_item1}$\Rightarrow$\eqref{proj_item2}$\Rightarrow$\eqref{proj_item3}. 

``\eqref{proj_item1}$\Rightarrow$\eqref{proj_item2}'' 
For proving \eqref{proj_item2a}, 
note that we always have $(G_0)_{\xi_0}\supseteq G_0\cap G_\xi$. 
For the converse inclusion let $g\in(G_0)_{\xi_0}$ arbitrary, 
hence $g\in G_0$ and $\xi_0\circ\Ad_{G_0}(g)=\xi_0$. 
Since $G_0\unlhd G$, it follows that $\gg_0\unlhd\gg$.  
Then we have 
$$\xi_0=\xi_0\circ\Ad_{G_0}(g)
=(\xi\vert_{\gg_0})\circ(\Ad_G(g)\vert_{\gg_0})
=(\xi\circ\Ad_G(g))\vert_{\gg_0}$$ 
hence we get 
$(\xi\circ\Ad_G(g))\vert_{\gg_0}=\xi\vert_{\gg_0}$.  
Since the restriction mapping $\rho\vert_{\Oc_\xi^G}$ is injective, 
it then follows that $\xi\circ\Ad_G(g)=\xi$, 
that is, $g\in G_\xi$. 
Therefore $g\in G_0\cap G_\xi$.  

For proving \eqref{proj_item2b}, let $g\in G$ arbitrary. 
We have $\rho(\Oc_\xi^G)=\Oc_{\xi_0}^{G_0}$, hence there exists $g_0\in G_0$ such that 
$(\xi\circ\Ad_G(g^{-1}))\vert_{\gg_0}=\xi_0\circ\Ad_{G_0}(g_0^{-1})$. 
Therefore $(\xi\circ\Ad_G(g^{-1}))\vert_{\gg_0}=(\xi\circ\Ad_G(g_0^{-1}))\vert_{\gg_0}$ 
and then by the hypothesis that $\rho\vert_{\Oc_\xi^G}$ is also injective 
we get $\xi\circ\Ad_G(g^{-1})=\xi\circ\Ad_G(g_0^{-1})$. 
This implies $\xi\circ\Ad_G(g_0^{-1}g)=\xi$, hence 
$g_0^{-1}g\in G_\xi$, that is, 
$g\in G_0G_\xi$. 
Thus $G\subseteq G_0G_\xi$, and the converse inclusion is obvious.

``\eqref{proj_item2}$\Rightarrow$\eqref{proj_item3}'' 
If both \eqref{proj_item2a} and \eqref{proj_item2b} hold true, 
then there exists the commutative diagram 
$$\xymatrix{
\Oc_\xi^G \ar[r]^{\rho\vert_{\Oc_\xi^G}} & \rho(\Oc_\xi^G) & \Oc_{\xi_0}^{G_0} \ar@{_{(}->}[l] \\
(G_0G_\xi)/G_\xi \ar[u] &  G_0/(G_0\cap G_\xi) \ar[l]_{\beta} & G_0/(G_0)_{\xi_0}  \ar[l]_{\hskip15pt \id} \ar[u]
}
$$
where the vertical arrows are the diffeomorphisms defined by the coadjoint actions of the groups 
$G$ ($=G_0G_\xi$) and $G_0$, respectively, 
and whose bottom arrow $\beta$ is given by $g(G_0\cap G_\xi)\mapsto gG_\xi$. 
It is straightforward to check that $\beta$ is a bijective immersion, 
hence either it is a diffeomorphism, or its differential fails to be surjective 
at every point. 
The latter variant is impossible because of Sard's theorem, 
hence $\beta$ is a diffeomorphism. 
(Alternatively, since $G_0\unlhd G$, it follows by the second isomorphism theorem for groups 
that we have a natural isomorphism of Lie groups  $G/G_0\simeq G_\xi/ (G_0\cap G_\xi)$, 
and this implies $\dim G-\dim G_\xi=\dim G_0-\dim(G_0\cap G_\xi)$. 
Then $\beta$ is a bijective immersion between manifolds with equal dimensions, 
hence $\beta$ is a diffeomorphism.) 

Note that the inclusion map $\rho(\Oc_\xi^G)\hookleftarrow\Oc_{\xi_0}^{G_0}$ in  the above diagram 
is actually bijective since 
for every $g\in G=G_0G_\xi$ there exists $g_0\in G_0$ such that $g\in g_0 G_\xi$,  
hence $\Ad_G^*(g)\xi=\Ad_G^*(g_0)\xi$, 
and then 
$$\rho(\Ad_G^*(g)\xi)=\rho(\Ad_G^*(g_0)\xi)=(\Ad_G^*(g_0)\xi)\vert_{\gg_0}
=\Ad_{G_0}^*(g_0)\xi_0\in\Oc_{\xi_0}^{G_0}.$$ 
We thus get the well-defined onto map $\rho\vert_{\Oc_\xi^G}\colon\Oc_\xi^G\to\Oc_{\xi_0}^{G_0}$. 
Since in the above commutative diagram the vertical arrows and the bottom ones are diffeomorphisms, 
it follows that $\rho\vert_{\Oc_\xi^G}\colon\Oc_\xi^G\to\Oc_{\xi_0}^{G_0}$ is also a diffeomorphism, 
and this concludes the proof. 
\end{proof}

Some implications from the statement of the following simple result 
can be obtained by the disintegration of restrictions of ireducible representations, 
but we give here an alternative, direct proof. 

\begin{proposition}\label{new}
Let $G=(\gg,\cdot)$ be a nilpotent Lie group with an irreducible representation $\pi\colon G\to\Bc(\Hc)$ 
associated with the coadjoint orbit $\Oc\subset\gg^*$. 
If we consider an ideal $\hg\unlhd\gg$ and the corresponding normal subgroup $H=(\hg,\cdot)$ of $G$, 
then the following assertions are equivalent: 
\begin{enumerate}
\item\label{new_item1} The restricted representation $\pi\vert_H\colon H\to\Bc(\Hc)$ is irreducible. 
\item\label{new_item2} The mapping $\gg^*\to\hg^*$, $\xi\mapsto\xi\vert_{\hg}$ 
gives a bijection of $\Oc$ onto a coadjoint orbit of $H$, which will be denoted $\Oc\vert_{\hg}$. 
\item\label{new_item3} For some/any $\xi_0\in\Oc$ we have $\gg=\gg_{\xi_0}+\hg$.
\item\label{new_item4} For some/any $\xi_0\in\Oc$ we have $\dim(\hg/(\gg_{\xi_0}\cap\hg))=\dim(\gg/\gg_{\xi_0})$. 
\item\label{new_item5} For some/any Jordan-H\"older sequence 
$\{0\}=\gg_0\subset\gg_1\subset\cdots\subset\gg_m=\gg$
with $\gg_k=\hg$ for $k=\dim\hg$, the corresponding set of jump indices of $\Oc$ is contained in $\{1,\dots,k\}$. 
\end{enumerate}
If this is the case, then the irreducible representation $\pi\vert_H$ is associated with 
the coadjoint orbit $\Oc\vert_{\hg}$ of $H$. 
\end{proposition}

\begin{proof}
First note that $\eqref{new_item3}\iff\eqref{new_item4}$ since 
there exists the canonical linear isomorphism $(\gg_{\xi_0}+\hg)/\gg_{\xi_0}\simeq\hg/(\gg_{\xi_0}\cap\hg)$, 
while $\eqref{new_item3}\iff\eqref{new_item5}$
by the definition of the jump indices. 

We will prove by induction on $\dim(\gg/\hg)$ that  $\eqref{new_item1}\iff\eqref{new_item2}\iff\eqref{new_item3}$. 
This follows at once by \cite[Th. 2.5.1(b)]{CG90} if $\dim(\gg/\hg)=1$, 
since in this case we have $\gg_{\xi_0}\not\subseteq\hg$ if and only if $\gg=\gg_{\xi_0}+\hg$. 
Now assume $\dim(\gg/\hg)\ge 2$. 
We will prove that 
$\eqref{new_item1}\Rightarrow\eqref{new_item2}\Rightarrow\eqref{new_item3}\Rightarrow\eqref{new_item1}$.
Since $\hg\unlhd\gg$, there exists a Jordan-H\"older sequence as in Assertion~\eqref{new_item5}. 
If \eqref{new_item1} holds true, then also the representation $\pi\vert_{G_{m-1}}\colon G_{m-1}\to\Bc(\Hc)$ 
is irreducible and its restriction to $H$ is $\pi\vert_H$. 
Therefore, by using \cite[Th. 2.5.1(b)]{CG90} we see that the restriction mapping $\gg^*\to\gg_{m-1}^*$ 
gives a bijection from $\Oc$ onto a coadjoint $G_{m-1}$-orbit $\Oc\vert_{\gg_{m-1}}$ 
of $G_{m-1}$. 
Then by the induction hypothesis it follows that 
the restriction mapping $\gg_{m-1}^*\to\hg^*$ gives a bijection 
from $\Oc\vert_{\gg_{m-1}}$ onto a coadjoint $H$-orbit $(\Oc\vert_{\gg_{m-1}})\vert_{\hg}$. 
Therefore Assertion~\eqref{new_item2} holds true. 
Moreover, if we know that Assertion~\eqref{new_item2} holds true, then 
Proposition~\ref{proj}\eqref{proj_item2a} shows that we have $\Oc\vert_{\hg}\simeq H/(H\cap G_{\xi_0})$, 
and then 
$$\dim(\gg/\gg_{\xi_0})=\dim\Oc=\dim(\Oc\vert_{\hg})=\dim(\hg/(\hg\cap\gg_{\xi_0}))$$
hence we get Assertion~\eqref{new_item4}, 
and we saw at the very beginning of the proof that $\eqref{new_item4}\iff\eqref{new_item3}$. 
Now assume that Assertion~\eqref{new_item3} holds true. 
Then we have 
$$\hg\subseteq\gg_{m-1}\subset\gg=\hg+\gg_{\xi_0}$$
hence $\gg_{\xi_0}\not\subseteq\gg_{m-1}$. 
Now, using \cite[Th. 2.5.1]{CG90} again 
along with the notation $G_{m-1}=(\gg_{m-1},\cdot)$,  
it follows that the representation $\pi\vert_{G_{m-1}}\colon G_{m-1}\to\Bc(\Hc)$ is irreducible. 
Furthermore, 
since $\hg\subset\gg_{m-1}\subset\hg+\gg_{\xi_0}$, it follows that 
$$\gg_{m-1}=\hg+(\gg_{m-1}\cap\gg_{\xi_0}).$$ 
If we denote $\xi_{0,m-1}:=\xi\vert_{\gg_{m-1}}$ then 
it is clear that $\gg_{m-1}\cap\gg_{\xi_0}$ is contained in the $G_{m-1}$-coadjoint isotropy algebra 
$$(\gg_{m-1})_{\xi_{0,m-1}}=\{x\in\gg_{m-1}\mid\langle\xi,[x,\gg_{m-1}]\rangle=\{0\}\}$$ 
hence by the above equality we get 
$\gg_{m-1}=\hg+(\gg_{m-1})_{\xi_{0,m-1}}$. 
Since $\dim(\gg_{m-1}/\hg)<\dim(\gg/\hg)$, 
we can apply the induction hypothesis for the irreducible representation $\pi\vert_{G_{m-1}}\colon G_{m-1}\to\Bc(\Hc)$. 
It thus follows that $(\pi\vert_{G_{m-1}})\vert_H$ is an irreducible representation. 
Since $(\pi\vert_{G_{m-1}})\vert_H=\pi\vert_H$, we see that Assertion~\eqref{new_item1} holds true, 
and this completes the induction step. 
\end{proof}

\section{Weyl-Pedersen calculus for some semidirect products}

We established in \cite{BB12} a quite natural version of the Calder\'on-Vaillancourt theorem 
for the Weyl-Pedersen calculus of the unitary irreducible representations of nilpotent Lie groups 
which are square integrable modulo the center of the group. 
In this section we use the above Proposition \ref{new} 
along with some results of \cite{BBP13} 
in order to prove that the same statement 
holds true for some unitary irreducible representations of more general nilpotent Lie groups, 
irrespective of whether these representations are square integrable or not. 

For the statement of Theorem~\ref{main3} we need the following remark. 

\begin{remark}\label{symp}
\normalfont
Let $(\gg_0,\omega)$ be any \emph{symplectic nilpotent Lie algebra}. 
This means that $\gg_0$ is a nilpotent Lie algebra and $\omega\colon\gg_0\times\gg_0\to\RR$ 
is a skew-symmetric nondegenerate bilinear functional satisfying the 2-cocycle condition
$$(\forall x,y,z\in\gg_0)\quad \omega(x,[y,z]_{\gg_0})+\omega(y,[z,x]_{\gg_0})+\omega(z,[x,y]_{\gg_0})=0. $$ 
Denote by $\gg:=\RR\dotplus_\omega\gg_0$ the corresponding 1-dimensional central extension, 
that is, $\gg=\RR\times\gg_0$ as a vector space and the Lie bracket of $\gg$ is defined by 
\begin{equation}\label{symp_eq0}
(\forall t,s\in\RR)(\forall x,y\in\gg_0)\quad [(t,x),(s,y)]_{\gg}:=(\omega(x,y),[x,y]_{\gg_0}). 
\end{equation} 
We denote by $G_0$ and $G$ the Lie groups obtained from $\gg_0$ and $\gg$ 
by using the multiplication given by the Baker-Campbell-Hausdorff formula. 
We define 
$$\Aut(\gg_0,\omega):=\{\alpha\in\Aut\gg_0\mid(\forall x,y\in\gg_0)\quad \omega(\alpha(x),\alpha(y))=\omega(x,y)\}$$
and we note that this has a natural embedding as a closed subgroup 
$$\Aut(\gg_0,\omega)\hookrightarrow\Aut\gg=\Aut G$$
since each $\alpha\in\Aut(\gg_0,\omega)$ can be extended to an automorphism $\alpha\in\Aut\gg$ with 
$\alpha(t,0)=(t,0)$ for all $t\in\RR$. 
A \emph{unipotent} automorphism group of $\gg_0$ is a closed subgroup 
$S\subseteq\Aut\gg_0$ with the property that 
for every $\alpha\in S$ there exists an integer $m\ge 1$ 
for which $(\alpha-\id_{\gg_0})^m=0$ on $\gg_0$. 

For any closed subgroup 
$S\subseteq\Aut(\gg_0,\omega)$, its Lie algebra is 
$$\sg:=\{D\in\Der(\gg_0,\omega)\mid(\forall t\in\RR)\ \exp(tD)\in S\}\hookrightarrow\Der\gg$$ 
where 
$$\Der(\gg_0,\omega):=\{D\in\Der\gg_0\mid(\forall x,y\in\gg_0)\quad \omega(Dx,y)+\omega(x,Dy)=0\}.$$
Each $D\in\Der(\gg_0,\omega)$ can be extended to a derivation $D\in\Der\gg$ 
with $D(t,0)=0$ for all $t\in\RR$. 
That is, if we denote by $\zg:=\RR\times\{0\}$ the center of $\gg$ ($=\RR\times\gg_0$), 
then we obtain a canonical isomorphism of Lie algebras 
\begin{equation}\label{symp_eq1}
\Der(\gg_0,\omega)\simeq\{D\in\Der\gg\mid \zg\subseteq\Ker D,\ \Ran D\subseteq\{0\}\times\gg_0\}. 
\end{equation} 
\end{remark}

\begin{theorem}\label{main3}
Let $(\gg_0,\omega)$ be any symplectic nilpotent Lie algebra.  
Define the nilpotent Lie algebra $\gg:=\RR\dotplus_\omega\gg_0$  
and let $G$ be its corresponding connected, simply connected nilpotent Lie group. 
Pick any connected, simply connected unipotent automorphism group $S\subseteq\Aut(\gg_0,\omega)$ 
and also define the corresponding semidirect product $\widetilde{G}:=S\ltimes G$. 

Then $\widetilde{G}$ is a connected, simply connected, nilpotent Lie group with 1-dimensional center. 
Moreover, if $\widetilde{\pi}\colon\widetilde{G}\to\Bc(\Hc)$ is any unitary irreducible representation 
associated with some coadjoint orbit $\widetilde{\Oc}$ for which 
there exists $\widetilde{\xi}\in\widetilde{\Oc}$ with 
$\widetilde{\xi}\vert_{\zg}\ne0$, $\widetilde{\xi}\vert_{\gg_0}=0$,  
and $\widetilde{\xi}\vert_{[\sg,\sg]}=0$, 
then for all $a\in C^\infty_b(\widetilde{\Oc})$ one has 
$$
(\forall D\in\Diff(\widetilde{\Oc})) \quad \Op(Da)\in\Bc(\Hc).$$ 
In addition, the Weyl-Pedersen calculus defines a continuous linear map 
$$\Op\colon \Cc^{\infty}_b(\widetilde{\Oc})\to\Bc(\Hc).$$ 
\end{theorem}

\begin{proof}
Recall that we denote the center of $\gg$ by $\zg$, and denote the center of $\widetilde{\gg}$ by~$\widetilde{\zg}$. 
The bracket in $\widetilde{\gg}$ is given by 
$$[(D_1,x_1),(D_2,x_2)]:=([D_1,D_2],D_1(x_2)-D_2(x_1)+[x_1,x_2])$$
for all $D_1,D_2\in\sg$ and $x_1,x_2\in\gg$. 
It follows by the above equality along with \eqref{symp_eq1} that $\{0\}\times\zg\subseteq\widetilde{\zg}$. 
Conversely, let $(D_1,x_1)\in\widetilde{\zg}$.  
Then by the above equality for $x_2=0$ we see that $D_1$ belongs to the center of $\sg$ 
and $x_1\in \bigcap\limits_{D\in\sg}\Ker D$. 
Moreover, by that equality with $D_2=0$ we obtain 
$$(\forall x_2\in\gg)\quad D_1(x_2)+[x_1,x_2]=0.$$ 
By writing $x_j=:(t_j,x_{j0})\in\RR\times\gg_0=\gg$ for $j=1,2$, 
one obtains by the above equation along with \eqref{symp_eq0} and \eqref{symp_eq1} that 
$$0=D_1(x_2)+[x_1,x_2]=(\omega(x_{10},x_{20}),D_1(x_{20})+[x_{10},x_{20}]_{\gg_0})$$ 
for all $x_{20}\in\gg_0$, hence $x_{10}=0$  (since $\omega$ is nondegenerate) 
and then $D_1=0$. 
Consequently $\widetilde{\zg}=\{0\}\times\zg$. 
Therefore we will henceforth write $\widetilde{\zg}=\zg$, 
and $\dim\zg=1$.

Since $\widetilde{\Oc}$ is a coadjoint orbit and $\widetilde{\xi}\in\widetilde{\Oc}$, 
it follows 
that $\langle \widetilde{\Oc},\zg\rangle=\langle \widetilde{\xi},\zg\rangle$. 
Now the hypothesis on $\widetilde{\xi}$ implies $\langle \widetilde{\Oc},\zg\rangle \ne\{0\}$,  
so $\langle \widetilde{\Oc},x_0\rangle=\{1\}$ for some $x_0\in\zg$. 

Now denote $\xi:=\widetilde{\xi}\vert_{\gg}\in\gg^*$ 
and $\Oc$ be the coadjoint $G$-orbit of $\xi$. 
Recalling the notation introduced in Proposition~\ref{new}\eqref{new_item2}, 
we now prove that 
\begin{equation}\label{main3_proof_eq0}
\widetilde{\Oc}\vert_{\gg}=\Oc. 
\end{equation}
In fact, one has the direct sums of vector spaces 
$\widetilde{\gg}=\sg\dotplus\gg$ and 
$\gg=\zg\dotplus\gg_0$ 
with the Lie brackets in $\widetilde{\gg}$ satisfying 
$$[\sg,\sg]\subseteq\sg,\ [\sg,\zg]=\{0\}, \text{ and }[\sg,\gg_0]\subseteq\gg_0,$$ 
hence the hypothesis  $\gg_0+[\sg,\sg]\subseteq\Ker\widetilde{\xi}$ implies $\langle\widetilde{\xi},[\sg,\widetilde{\gg}]\rangle=\{0\}$, 
that is, $\sg\subseteq\widetilde{\gg}_{\widetilde{\xi}}$. 
It then follows that $\widetilde{\gg}=\widetilde{\gg}_{\widetilde{\xi}}+\gg$ 
and now \eqref{main3_proof_eq0} follows by the implication 
\eqref{new_item3}$\Rightarrow$\eqref{new_item2} in Proposition~\ref{new}. 

Since $\widetilde{\Oc}\vert_{\gg}=\Oc$, 
it follows by the implication \eqref{new_item2}$\Rightarrow$\eqref{new_item1} 
in Proposition~\ref{new} that $\widetilde{\pi}\vert_G\colon G\to\Bc(\Hc)$ 
is an unitary irreducible representation associated with $\Oc$.
If we introduce the $G$-equivariant diffeomorphism $\rho_G\colon\widetilde{\Oc}\to\Oc$, $\xi\mapsto\xi\vert_{\gg}$, 
then by \cite[(3.2)]{BBP13} we have 
\begin{equation}\label{main3_proof_eq1}
(\forall a\in\Sc'(\Oc))\quad \Op_{\widetilde{\pi}}(a)=\Op_{\pi}(a\circ \rho_G^{-1})
\end{equation} 
and one has an injective homomorphism 
of associative algebras 
$$\rho_G^*\colon\Diff(\widetilde{\Oc})\to\Diff(\Oc). $$
Just as in \cite[Rem. 3.3]{BBP13} we have 
$D(a\circ\rho_G)=(\rho_G^*(D)a)\circ\rho_G$ for all $a\in\Ci(\Oc)$ 
and $D\in\Diff(\widetilde{\Oc})$, 
hence one has an injective map 
$\Cc^{\infty}_b(\widetilde{\Oc})\hookrightarrow\Cc^{\infty}_b(\Oc)$, $a\mapsto a\circ\rho_G^{-1}$. 

Finally, since $G$ is a nilpotent Lie group whose coadjoint orbit $\Oc$ 
is flat,  by \cite[Th. 1.1]{BB12} 
for $a\in C^\infty(\Oc)$ we have 
\begin{equation}\label{main3_proof_eq2}
a\in \Cc^\infty_b(\Oc)\implies
(\forall D\in\Diff(\Oc)) \quad \Op_{\pi}(Da)\in\Bc(\Hc),
\end{equation} 
and 
$$\Op_{\pi}\colon \Cc^{\infty}_b(\Oc)\to\Bc(\Hc)$$
is continuous. 
The above remarks then show that these assertions hold true for $\Op_{\widetilde{\pi}}$, and this concludes the proof. 
\end{proof}

Now we show that Theorem~\ref{main3} is applicable for every unitary irreducible representation 
of any connected, simply connected, 
3-step nilpotent Lie group with 1-dimensional center. 

\begin{corollary}\label{3step}
Let $\widetilde{G}$ be any connected, simply connected, 3-step nilpotent Lie group with its Lie algebra $\widetilde{\gg}$. 
If $\pi\colon\widetilde{G}\to\Bc(\Hc)$ is any unitary irreducible representation 
associated with some coadjoint orbit $\widetilde{\Oc}\subseteq\widetilde{\gg}^*$, 
then 
for all $a\in C^\infty_b(\widetilde{\Oc})$ one has 
$$
(\forall D\in\Diff(\widetilde{\Oc})) \quad \Op(Da)\in\Bc(\Hc).$$ 
In addition, the Weyl-Pedersen calculus defines a continuous linear map 
$$\Op\colon \Cc^{\infty}_b(\widetilde{\Oc})\to\Bc(\Hc).$$
\end{corollary}

\begin{proof}
By standard arguments, one may assume that the center of $\widetilde{G}$ is 1-dim\-ensional. 
Then $\widetilde{\gg}$ is a 3-step nilpotent Lie algebra whose center~$\zg$ is 1-dimensional. 
It follows by \cite[Th. 5.1]{BB13} that there exists a decomposition 
$$\widetilde{\gg}=\hg\widetilde{\times}\bigl((\zg\dotplus\cg\dotplus\Vc)\widetilde{\rtimes}(\hg_1\dotplus\sg)\bigr) $$
 some subalgebras $\hg$, $\cg$, $\hg_1$, $\ag$, and a linear subspace $\Vc$ of $\widetilde{\gg}$. 
 Here we use the operations of reduced direct product $\widetilde{\times}$ and reduced semidirect product $\widetilde{\rtimes}$ 
 introduced in \cite[Def. 2.10]{BB13} on Lie algebras with 1-dimensional centers, 
 and the following conditions are satisfied: 
\begin{itemize}
\item $\widetilde{\gg}=\hg+\zg+\cg+\Vc+\hg_1+\sg$; 
\item $\hg$ and $\hg_1$ are Heisenberg algebras that contain $\zg$, and $[\hg,\widetilde{\gg}]\subseteq\zg$; 
\item $\hg\cap(\zg+\cg+\Vc+\hg_1+\sg)=\zg$, 
$(\zg+\cg+\Vc)\cap(\hg_1+\sg)=\zg$;  
\item $\cg$ is an abelian subalgebra, $[\Vc,\Vc]\subseteq\cg$, 
and $[\cdot,\cdot]\colon\Vc\times\cg\to\zg$ is a nondegenerate bilinear map 
(hence in particular $\dim\Vc=\dim\cg$, since $\zg\simeq\RR$); 
\item $[\hg_1\dotplus\sg,\cg]=\{0\}$ and $[\hg_1\dotplus\ag,\Vc]\subseteq\cg$; 
\item $[\hg_1,\sg]=\{0\}$. 
\end{itemize}
If we define 
$$\gg:=\hg\widetilde{\times}\bigl((\zg\dotplus\cg\dotplus\Vc)\widetilde{\rtimes}\hg_1\bigr) $$
then we obtain $\widetilde{\gg}=\gg\dotplus\sg$, $[\sg,\gg]\subseteq\gg$.  
Moreover it 
follows by \cite[Th. 5.2(1)]{BB13} that the connected simply connected Lie group  
associated with $(\zg\dotplus\cg\dotplus\Vc)\widetilde{\rtimes}\hg_1$
has flat generic coadjoint orbits. 
It is easily checked that the operation of reduced direct product 
preserves the class of nilpotent Lie algebras with 1-dimensional center and generic flat coadjoint orbits,
\footnote{Let $\kg$ be any nilpotent Lie algebra whose center $\zg$ is 1-dimensional, 
and $\kg_1$ and $\kg_2$ be subalgebras of $\kg$ with generic flat coadjoint orbits and 1-dimensional centers, 
such that $\kg=\kg_1\widetilde{\times}\kg_2$,  
that is, $\kg=\kg_1+\kg_2$, $\zg=\kg_1\cap\kg_2$, and $[\kg_1,\kg_2]=\{0\}$. 
For $j=1,2$ pick any linear subspace $\kg_j^0\subseteq\kg_j$ with $\zg\dotplus\kg_j^0=\kg_j$. 
We claim that for every $\xi\in\kg^*$ with $\Ker\xi=\kg_1^0\dotplus\kg_2^0$ one has $\kg_\xi=\zg$, 
or, equivalently, for every $X\in\kg\setminus\zg$ there exists $Y\in\kg$ 
with $[X,Y]\not\in\Ker\xi$. 
To check that condition, let $X\in\kg=\zg\dotplus\kg_1^0\dotplus\kg_2^0$ with $X\not\in\zg$, 
hence $X\in X_1^0+X_2^0+\zg$, where $X_j^0\in\kg_j^0$ for $j=1,2$ with $X_1^0+X_2^0\ne0$. 
If $j\in\{1,2\}$ and  
$X_j^0\ne 0$, use the hypothesis that the generic coadjoint orbits of $\kg_j$ are flat and 
$\xi\vert_{\kg_j}$ does not vanish identically on the center $\zg$ of $\kg_j$, 
to find $Y_j\in\kg_j$ with $\langle\xi ,[X_j^0,Y_j]\rangle=1$. 
If $X_j^0=0$ then let $Y_j^0:=0$.
By the hypothesis $[\kg_1,\kg_2]=\{0\}$ we obtain $\langle\xi,[X,Y_1+Y_2]\rangle\ge1$, 
hence $[X,Y_1+Y_2]\not\in\Ker\xi$.} 
hence also the connected simply connected Lie group 
$G$ associated to $\gg$ has generic flat coadjoint orbits. 
$$\gg_0:=(\hg/\zg)\times\bigl((\cg\dotplus\Vc)\rtimes(\hg_1/\zg)\bigr) $$
$S$ as the connected subgroup of $\Aut\gg_0$ defined  by integrating 
the representation of $\sg$ on $\gg_0$ coming from the commutation relation $[\sg,\gg]\subseteq\gg$, 
then we see that all the hypotheses of Theorem~\ref{main3} are satisfied. 
\end{proof}

\begin{remark}\label{symp1}
 \normalfont 
In connection with other situations where Theorem~\ref{main3} is applicable, we note the following. 
If $(\gg_0,\omega)$ is any symplectic nilpotent Lie algebra, 
it is easily checked that if a closed subgroup $S\subseteq\Aut(\gg_0,\omega)$ is connected, 
then $S$ is unipotent if and only if  
$\sg$ consists of nilpotent derivations of $\gg_0$. 
This remark leads directly to two types of examples of such situations: 
\begin{enumerate}
\item 
If $\gg_0$ is a characteristically nilpotent Lie algebra, 
then $\Aut\gg_0$ is unipotent, and so are all its connected closed subgroups. 
Characteristically nilpotent Lie algebras with symplectic structures 
were studied in \cite{Bu06}. 
\item Let $\pg$ be any polarization of $(\gg_0,\omega)$, 
hence $\pg$ is a subalgebra of $\gg_0$ with $\omega\vert_{\pg\times\pg}=0$ and $\dim\gg_0=2\dim\pg$. 
If we define 
$$S_{\pg}:=\{\alpha\in\Aut(\gg_0,\omega)\mid (\forall x\in\pg)\quad \alpha(x)=x\}$$
then $S_{\pg}$ is a unipotent automorphism group of $\gg_0$ and 
for every $\alpha\in S_{\pg}$ we have $(\alpha-\id_{\gg_0})^2=0$. 
This follows from the corresponding statement for an abelian algebra $\gg_0=\RR^{2n}$, 
and in that special case $S_{\pg}$ is just the group $S_n$ from \cite[Sect. 3]{Ra85}. 
\end{enumerate} 
\end{remark}

\begin{problem}
\normalfont
It would be interesting to know whether some version of Beals' commutator criterion 
for recognizing pseudo-differential operators can be established for the Weyl-Pedersen calculus  
associated with coadjoint orbits that are not flat. 
Regarding the Weyl-Pedersen calculus on flat coadjoint orbits, we recall from \cite[Th. 1.1]{BB12} 
that Beals' criterion holds true under the form of the converse of the implication~\eqref{main3_proof_eq2} 
from the proof of Theorem~\ref{main3} above. 
\end{problem}

\section{Specific examples}

We now discuss some situations 
considered in the earlier literature, in particular providing
an uncountable family of pairwise nonisomorphic examples of 3-step nilpotent Lie algebras with 1-dimensional centers. 
(Recall that there exist only countably many isomorphism classes of 2-step nilpotent Lie algebras with 1-dimensional centers, 
since these are precisely the Heisenberg algebras, 
hence there exists precisely one isomorphism class for every odd dimension, 
and no isomorphism classes for the even dimensions.)

\begin{example}[{\cite[Ex. 3.5]{La05}, \cite[Ex. 5.2]{La06}}]
\normalfont
For all $s,t\in\RR\setminus\{0\}$ let $\gg_0(s,t)$ be the 6-dimensional 2-step nilpotent Lie algebra 
defined by the commutation relations 
$$[X_6,X_5]=sX_3,\ [X_6,X_4]=(s+t)X_2,\ [X_5,X_4]=tX_1. $$
It follows by \cite[Ex. 3.5]{La05} that 
$$\{\gg_0(s,t)\mid s^2+st+t^2=1,\ 0<t\le 1/\sqrt{3}\}$$
is a family of isomorphic Lie algebras 
that are however pairwise non-isomorphic as symplectic Lie algebras 
with 
the common symplectic structure $\omega$ given by the matrix 
$$J_{\omega}=
\begin{pmatrix}
\hfill 0 &\hfill 0 &\hfill 0 & 0 & 0 & 1 \\
\hfill 0 &\hfill 0 &\hfill 0 & 0 & 1 & 0 \\
\hfill 0 &\hfill 0 &\hfill 0 & 1 & 0 & 0 \\
\hfill 0 &\hfill 0 & -1      & 0 & 0 & 0 \\
\hfill 0 & -1      &\hfill 0 & 0 & 0 & 0 \\
      -1 &\hfill 0 &\hfill 0 & 0 & 0 & 0 
\end{pmatrix}$$
Then it is easily checked that for every symmetric matrix $A=A^\top\in M_3(\RR)$ 
the matrix 
$$D_A:=
\begin{pmatrix}
0 & A \\
0 & 0
\end{pmatrix}\in M_6(\RR)$$
satisfies $J_\omega D_A+D_A^TJ_\omega=0$, $D_A^2=0$, and $D_A\in\Der(\gg_0(s,t))$ for all $s,t\in\RR\setminus\{0\}$. 

Then, as in Example~\ref{N5N3_plus}, 
the group $S_A:=\{\exp(tD_A)\mid t\in\RR\}\subseteq\Aut(\gg_0(s,t),\omega)$ satisfies the hypothesis of Theorem~\ref{main3} 
for all $s,t\in\RR\setminus\{0\}$ and $A=A^\top\in M_3(\RR)$.  
If we define $\gg(s,t):=\RR\dotplus_\omega\gg_0(s,t)$ and $\widetilde{\gg}(s,t,A):=S_A\ltimes \gg(s,t)$, 
then Theorem~\ref{main3} can be applied for each Lie group in the family 
$$\{\widetilde{G}(s,t,A)\mid s,t\in\RR\setminus\{0\},\ A=A^\top\in M_3(\RR)\}.$$ 
\end{example}

\begin{example}[{\cite[page 560]{Pe89}}]\label{N5N3_plus}
\normalfont 
Let $\widetilde{\gg}$ be the 6-dimensional Lie algebra 
defined by the commutation relations 
$$\begin{aligned}
& [X_6,X_5]=X_4,\ [X_6,X_4]=X_3,\ [X_6,X_3]=X_2, \\ 
& [X_5,X_4]=X_2,\ [X_5,X_2]=-X_1, \\
& [X_4,X_3]=X_1.
\end{aligned}$$
Then $\widetilde{\gg}$ is 4-step nilpotent with 1-dimensional center. 
If we define 
$$\gg:=\spa\{X_1,X_2,X_3,X_4,X_5\}$$
then it is easily checked that the linear map defined by 
$$X_3\mapsto -X_2,\  X_1\mapsto -X_0,\  
X_j\mapsto X_{j-1}\text{ for }j\in\{2,4,5\},$$ 
is a Lie algebra isomorphism of $\gg$ onto the Lie algebra (denoted also by $\gg$) in \cite[Ex. 5.7]{BB12}. 
The center $\zg$ of $\gg$ is 1-dimensional, specifically $\zg=\RR X_1$. 
Therefore the Lie algebra $\gg_0=\gg/\zg=\spa\{X_2,X_3,X_4,X_5\}$ is defined by the commutation relation 
$$[X_5,X_4]=X_2$$ 
and the center of $\gg_0$ is spanned by $\{X_2,X_3\}$. 
The skew-symmetric bilinear functional $\omega\colon\gg_0\times\gg_0\to\RR$ is 
defined by  
$\omega(X_2,X_5)=\omega(-X_3,X_4)=-1$ and $\omega(X_i,X_j)=0$ if $2\le i<j\le 5$ and 
$(i,j)\not\in\{(2,5),(3,4)\}$, 
and this corresponds to the matrix 
$$J_\omega=
\begin{pmatrix}
 0 &\hfill 0 & 0 & -1 \\
 0 &\hfill 0 & 1 &\hfill 0 \\
 0 &      -1 & 0 &\hfill 0 \\
 1 &\hfill 0 & 0 &\hfill 0 
\end{pmatrix}$$
while $\gg$ is an ideal of $\widetilde{\gg}$ and $D:=(\ad_{\widetilde{\gg}}X_6)\vert_{\gg}\colon\gg\to\gg$ 
is a derivation that vanishes on~$\zg$, 
hence it induces a derivtion of $\gg/\zg=\gg_0$ that is given by the matrix 
$$D=
\begin{pmatrix}
0 & 1 & 0 & 0\\
0 & 0 & 1 & 0\\
0 & 0 & 0 & 1 \\
0 & 0 & 0 & 0
\end{pmatrix}$$
The above matrices satisfy the equation $J_\omega D+D^\top J_\omega=0$, 
which is equivalent to $\omega(Dx,y)+\omega(x,Dy)=0$ for all $x,y\in\gg_0$, 
and further to $\exp(tD)\in\Aut(\gg_0,\omega)$ for all $t\in\RR$. 
Thus the group $S:=\{\exp(tD)\mid t\in\RR\}$ satisfies the hypothesis of Theorem~\ref{main3} 
and we have $\widetilde{G}=S\ltimes G$, where $\widetilde{G}$ and $G$ 
are the connected simply conncted Lie groups that correspond to the Lie algebras $\widetilde{\gg}$ and $\gg$, 
respectively. 

Now let $\widetilde{\xi}\in\widetilde{\gg}^*$ with $\langle\widetilde{\xi},X_1\rangle=:a\ne0$, $\langle\widetilde{\xi},X_j\rangle=0$ for $j=2,3,4,5$ 
(that is, $\gg_0\subseteq\Ker\widetilde{\xi}$), 
and $\langle\widetilde{\xi},X_6\rangle=:b$, 
and denote by $\widetilde{\Oc}$ the coadjoint $\widetilde{G}$-orbit of $\widetilde{\xi}$. 
It follows by \cite[page 562]{Pe89} that $\widetilde{\Oc}$ is given by the equation 
$$y_6=\frac{1}{6}(6a^2b+6ay_2y_4-3ay_3^2+2y_2^3) $$
where $(y_1,y_2,y_3,y_4,y_5,y_6)$ are the Cartesian coordinates in $\widetilde{\gg}^*$ 
with respect to the dual basis 
of $\{X_1,X_2,X_3,X_4,X_5,X_6\}$. 

As indicated also on \cite[page 562]{Pe89}, an irreducible representation  
associated with $\widetilde{\Oc}$ is $\pi\colon \widetilde{G}\to\Bc(L^2(\RR^2))$ given by 
$$\begin{aligned}
\de\pi(X_1) &= \ie a\\
\de\pi(X_2) &= \ie t_1\\
\de\pi(X_3) &= \ie t_2\\
\de\pi(X_4) &= a\frac{\partial}{\partial t_2}\\
\de\pi(X_5) &= \frac{1}{a}(-a^2\frac{\partial}{\partial t_1}-\ie t_1 t_2)\\
\de\pi(X_6) &= \frac{1}{6a^2}(6a^2 t_1\frac{\partial}{\partial t_2}+6\ie ab-3\ie at_2^2+2\ie t_1^3)
\end{aligned}
$$
where $(t_1,t_2)$ are the Cartesian coordinates in $\RR^2$. 
\end{example}

\end{document}